\documentclass[12pt]{amsart}
\usepackage[a4paper,margin=1in]{geometry}
\usepackage{amsmath,amssymb,amsthm}
\usepackage{xcolor}
\usepackage{tikz}
\usetikzlibrary{calc}
\usepackage[colorlinks=true,linkcolor=red,citecolor=blue,urlcolor=blue]{hyperref}

\newtheorem{theorem}{Theorem}[section]
\newtheorem{lemma}[theorem]{Lemma}
\newtheorem{claim}{Claim}

\title{The Borodin--Kostochka Conjecture for Two Classes of Diamond-Free Graphs}
\author{Alireza Fiujlaali}
\email{afiujlaali1@gsu.edu}
\date{}

\begin{document}

\begin{abstract}
We prove the Borodin--Kostochka conjecture for $(\mathrm{diamond},L_4)$-free graphs and $(\mathrm{diamond},K_5^+)$-free graphs, where $L_4$ is obtained from $K_4$ by attaching a path of length two to one clique vertex, and $K_5^+$ is obtained from $K_5$ by adding a vertex adjacent to exactly one clique vertex. The proof uses structural properties of maximum cliques in diamond-free vertex-critical graphs with $\chi=\Delta$.
\end{abstract}

\maketitle

\section{Introduction}

All graphs are finite and simple. For a graph $G$, let $V(G)$ and $E(G)$ denote its vertex and edge sets, and let $\Delta(G)$, $\delta(G)$, $\chi(G)$, and $\omega(G)$ denote its maximum degree, minimum degree, chromatic number, and clique number, respectively. For $X\subseteq V(G)$, let $G[X]$ denote the subgraph induced by $X$. A graph is \textit{$H$-free} if it contains no induced copy of $H$, and is \textit{$k$-vertex-critical} if $\chi(G)=k$ and $\chi(G-v)<k$ for every $v\in V(G)$. A set $Q\subseteq V(G)$ is \textit{dominating} if every vertex outside $Q$ has a neighbor in $Q$. The \textit{diamond} is $K_4-e$. For $t\ge3$, let $L_t$ be obtained from $K_t$ by attaching a path of length two to one clique vertex, and let $K_t^+$ be obtained from $K_t$ by adding one vertex adjacent to exactly one clique vertex.

Brooks' theorem \cite{Brooks} gives the classical bound $\chi(G)\le\Delta(G)$ for connected graphs that are neither complete nor odd cycles. In 1977, Borodin and Kostochka \cite{BK} conjectured that if $\Delta(G)\ge9$, then
$$
\chi(G)\le\max\{\Delta(G)-1,\omega(G)\}.
$$
They proved that every graph with $\chi(G)=\Delta(G)\ge7$ contains a clique of order $\lfloor(\Delta(G)+1)/2\rfloor$. Catlin \cite{Catlin} and Kostochka \cite{Kostochka} independently showed that the conjecture reduces to maximum degree nine; the hereditary-class formulation needed here is recorded by Cranston and Rabern \cite{CR}. Reed \cite{Reed} proved the conjecture for sufficiently large maximum degree, and Cranston and Rabern \cite{CR,CRbig} obtained several equivalent formulations and structural results for graphs with $\chi=\Delta$. More recently, Galindo and McDonald \cite{GM} studied directly the case $\chi=\Delta=9$.

The conjecture has also been verified for several hereditary classes, including claw-free graphs \cite{CRclaw}, $(P_5,C_4)$-free graphs \cite{GP21}, $(P_5,\mathrm{gem})$-free graphs \cite{CLR}, hammer-free graphs \cite{Hammer}, and odd-hole-free graphs \cite{OddHole}. For classes involving the diamond, Goedgebeur, Huang, Ju and Merkel \cite{GHJM} proved $\chi(G)\le\max\{6,\omega(G)\}$ for $(P_6,\mathrm{diamond})$-free graphs, Lan, Zhou and Liu \cite{LZL} proved $\chi(G)\le\max\{3,\omega(G)\}$ for $(P_6,C_4,\mathrm{diamond})$-free graphs, Gupta and Pradhan \cite{GP24} obtained Borodin--Kostochka type bounds for several $P_6$-free classes, and Wu and Yang \cite{WuYang} proved the conjecture for $(\mathrm{bull},\mathrm{diamond})$-free graphs.

Fiujlaali \cite{Fiujlaali} proved the conjecture for $(\mathrm{diamond},F_t)$-free graphs with $\Delta(G)\ge 2t+1$, where $F_t$ is obtained from $K_t$ by adding one pendant vertex to each clique vertex. The class considered there differs from the two classes considered here, since $F_4$ contains neither $L_4$ nor $K_5^+$. Our main result is the following.

\begin{theorem}\label{thm:BK}
Let $G$ be a diamond-free graph with $\Delta(G)\ge9$. If $G$ is $L_4$-free or $K_5^+$-free, then $\chi(G)\le\max\{\Delta(G)-1,\omega(G)\}$.
\end{theorem}

The proof is based on the following structural theorem for diamond-free critical graphs.

\begin{theorem}\label{thm:main}
Let $G$ be a noncomplete $k$-vertex-critical diamond-free graph with $\chi(G)=\Delta(G)=k$ and $\omega(G)=r\ge4$. Then $G$ contains an induced $K_s^+$ for every $3\le s\le r$ and an induced $L_t$ for every $3\le t\le r-1$. If $G$ is $L_r$-free, then every maximum clique is dominating.
\end{theorem}

Theorem~\ref{thm:main} applies to noncomplete diamond-free critical graphs with $\chi=\Delta$ and $\omega\ge4$, without requiring a second forbidden induced graph. When a maximum clique is dominating, we obtain a more precise description of the vertices outside it. In particular, after grouping these vertices according to their unique neighbor in the clique and then into clique components, the resulting auxiliary graph has at least $k-r+2$ vertices, and we characterize the equality case. The classes in Theorem~\ref{thm:BK} are also different from the previously studied classes above. For example, $(\mathrm{diamond},L_4)$-free graphs and $(P_6,\mathrm{diamond})$-free graphs are incomparable; $P_6$ belongs to the former class and $L_4$ belongs to the latter. Moreover, $(\mathrm{diamond},L_4)$-free graphs may contain a bull or a hammer. Thus Theorem~\ref{thm:BK} is not a formal consequence of the results in \cite{GHJM,WuYang,Hammer}.

\begin{figure}[t]
\centering
\begin{tikzpicture}[scale=.63,every node/.style={circle,fill,inner sep=1.25pt}]
\begin{scope}[xshift=0cm,yshift=0cm]
\node (c) at (.5,.35) {}; \node (a) at (-.15,1.05) {}; \node (b) at (1.15,1.05) {}; \node (d) at (.5,-.65) {};
\draw (c)--(a); \draw (c)--(b); \draw (c)--(d);
\node[draw=none,fill=none,rectangle] at (.5,-1.25) {claw};
\end{scope}
\begin{scope}[xshift=3.3cm,yshift=0cm]
\node (p1) at (0,0) {}; \node (p2) at (.7,0) {}; \node (p3) at (1.4,0) {}; \node (p4) at (2.1,0) {}; \node (u) at (1.05,1.0) {};
\draw (p1)--(p2)--(p3)--(p4); \draw (u)--(p1); \draw (u)--(p2); \draw (u)--(p3); \draw (u)--(p4);
\node[draw=none,fill=none,rectangle] at (1.05,-.75) {gem};
\end{scope}
\begin{scope}[xshift=7.0cm,yshift=0cm]
\node (a) at (0,0) {}; \node (b) at (1,0) {}; \node (c) at (.5,.85) {}; \node (d) at (.5,-.85) {};
\draw (a)--(b); \draw (a)--(c)--(b); \draw (a)--(d)--(b);
\node[draw=none,fill=none,rectangle] at (.5,-1.45) {diamond};
\end{scope}
\begin{scope}[xshift=10.2cm,yshift=0cm]
\node (a) at (0,0) {}; \node (b) at (1,0) {}; \node (c) at (.5,.85) {}; \node (x) at (-.65,-.65) {}; \node (y) at (1.65,-.65) {};
\draw (a)--(b)--(c)--(a); \draw (a)--(x); \draw (b)--(y);
\node[draw=none,fill=none,rectangle] at (.5,-1.45) {bull};
\end{scope}
\begin{scope}[xshift=0cm,yshift=-4.0cm]
\node (a) at (0,0) {}; \node (b) at (1,0) {}; \node (c) at (.5,.85) {}; \node (x) at (-.8,0) {}; \node (y) at (-1.55,0) {};
\draw (a)--(b)--(c)--(a); \draw (a)--(x)--(y);
\node[draw=none,fill=none,rectangle] at (.2,-.75) {hammer};
\end{scope}
\begin{scope}[xshift=10.2cm,yshift=-4.0cm]
\node (a) at (0,0) {}; \node (b) at (1,0) {}; \node (c) at (1,.9) {}; \node (d) at (0,.9) {}; \node (x) at (-.8,.45) {};
\draw (a)--(b)--(c)--(d)--(a); \draw (a)--(c); \draw (b)--(d); \draw (x)--(d);
\node[draw=none,fill=none,rectangle] at (.5,-.7) {$K_4^+$};
\end{scope}
\begin{scope}[xshift=5.2cm,yshift=-4.0cm]
\node (a) at (0,0) {}; \node (b) at (1,0) {}; \node (c) at (1,.9) {}; \node (d) at (0,.9) {}; \node (x) at (-.8,.45) {}; \node (y) at (-1.6,.45) {};
\draw (a)--(b)--(c)--(d)--(a); \draw (a)--(c); \draw (b)--(d); \draw (d)--(x)--(y);
\node[draw=none,fill=none,rectangle] at (.5,-.7) {$L_4$};
\end{scope}
\end{tikzpicture}
\caption{Some small graphs appearing in the paper and related results.}
\label{fig:forbidden}
\end{figure}
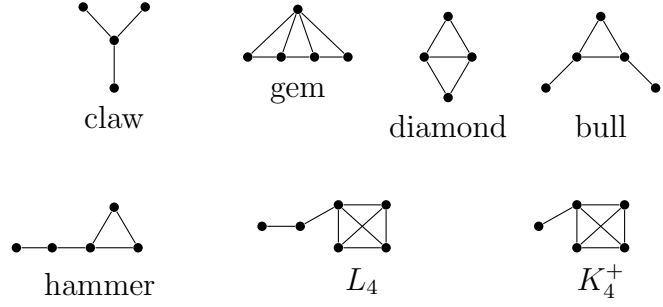

\section{Preliminaries}\label{sec:prelim}

 For a vertex $v$, let $N(v)$ denote its neighborhood, and for $X\subseteq V(G)$ let $N_X(v)=N(v)\cap X$. A \textit{block} of a graph is a maximal connected subgraph with no cut-vertex. A \textit{clique cutset} is a clique whose deletion disconnects the graph. A class of graphs is \textit{hereditary} if it is closed under taking induced subgraphs. We first state the following results used in the proofs.

\begin{theorem}[Brooks \cite{Brooks}]\label{thm:brooks}
Every graph $G$ with $\Delta(G)\ge3$ satisfies
$$
\chi(G)\le\max\{\omega(G),\Delta(G)\}.
$$
\end{theorem}

\begin{theorem}[Borodin--Kostochka \cite{BK}]\label{thm:BKclique}
Every graph $G$ with $\chi(G)=\Delta(G)\ge7$ contains $K_{\lfloor(\Delta(G)+1)/2\rfloor}$.
\end{theorem}

\begin{lemma}[Catlin--Kostochka \cite{Catlin,Kostochka,CR}]\label{lem:reduction}
Let $\mathcal H$ be a hereditary class of graphs and let $k\ge5$. If $\mathcal H$ contains no $k$-vertex-critical graph $G$ with $\chi(G)=\Delta(G)=k$ and $\omega(G)<k$, then $\mathcal H$ contains no $(k+1)$-vertex-critical graph $G$ with $\chi(G)=\Delta(G)=k+1$ and $\omega(G)<k+1$. In particular, to prove the Borodin--Kostochka conjecture for a hereditary class, it is enough to prove it when $\Delta=9$.
\end{lemma}

\begin{lemma}[Gallai \cite{Gallai}]\label{lem:gallai}
In a $k$-vertex-critical graph, the subgraph induced by the vertices of degree $k-1$ is a Gallai forest; that is, every block is a clique or an odd cycle.
\end{lemma}

We will also use a few elementary consequences of vertex-criticality. Every graph with chromatic number $k$ contains a $k$-vertex-critical induced subgraph, obtained by deleting vertices as long as the chromatic number remains $k$. If $G$ is $k$-vertex-critical, then $G$ is connected and $\delta(G)\ge k-1$; otherwise a $(k-1)$-coloring of a proper component or of $G-v$ extends to all of $G$. A vertex-critical graph also has no clique cutset. Indeed, if a clique $S$ separates a $k$-vertex-critical graph, the proper induced subgraphs obtained from two sides of $G-S$ together with $S$ are both $(k-1)$-colorable. Since $S$ is a clique, the colors on one side can be permuted so that the two colorings agree on $S$, giving a $(k-1)$-coloring of $G$.

\begin{lemma}\label{lem:diamond}
Let $G$ be diamond-free. Every component of $G[N(v)]$ is a clique. Moreover, a vertex outside a maximum clique of order at least three has at most one neighbor in that clique.
\end{lemma}

\begin{proof}
Suppose first that some component of $G[N(v)]$ is not a clique. Then it contains two nonadjacent vertices joined by a shortest path inside $G[N(v)]$. The first three vertices of this path induce a $P_3$. Since all three are neighbors of $v$, adding $v$ gives a diamond, a contradiction. Thus every component of $G[N(v)]$ is a clique.

Now let $Q$ be a maximum clique of order at least three, and suppose that a vertex $x\notin Q$ has two neighbors $q_i,q_j\in Q$. The vertex $x$ cannot be adjacent to every vertex of $Q$, since otherwise $Q\cup\{x\}$ would be a larger clique. Hence $x$ misses some $q_\ell\in Q$. The three vertices $q_i,q_j,q_\ell$ form a triangle, while $x$ is adjacent to $q_i$ and $q_j$ but not to $q_\ell$. Therefore $x,q_i,q_j,q_\ell$ induce a diamond, again a contradiction.
\end{proof}

For the remainder of this section, let $G$ be a noncomplete $k$-vertex-critical diamond-free graph with $\chi(G)=\Delta(G)=k$ and $\omega(G)=r\ge4$, and let $Q=\{q_1,\ldots,q_r\}$ be a dominating maximum clique. Since $G$ is noncomplete and $k$-vertex-critical, it contains no $K_k$, and hence $r\le k-1$. Since every vertex has degree at least $k-1$, each $q_i$ has a neighbor outside $Q$. By Lemma~\ref{lem:diamond}, every vertex outside $Q$ has exactly one neighbor in $Q$. For $1\le i\le r$, let $R_i$ consist of the vertices outside $Q$ whose neighbor in $Q$ is $q_i$. Thus $R_1,\ldots,R_r$ are nonempty and partition $V(G)\setminus Q$. Let $H$ be the graph whose vertices are the components of the graphs $G[R_i]$, with two components adjacent in $H$ whenever an edge of $G$ joins them.

\begin{lemma}\label{lem:structure}
The following hold.
\begin{enumerate}
\item $H$ is connected.
\item Every $C\in V(H)$ is a clique of order at most $r-1$, and the edges of $G$ between any two components of $H$ form a matching.
\item Every $C\in V(H)$ satisfies $d_H(C)\ge k-1-|C|$.
\end{enumerate}
\end{lemma}

\begin{proof}
We prove the three assertions in order.

For (1), the elementary observation above shows that a vertex-critical graph has no clique cutset. Since $Q$ is a clique, $G-Q$ must therefore be connected. The vertices of $H$ are the components into which the sets $R_i$ split, and two such components are adjacent in $H$ exactly when an edge of $G$ joins them. Thus a separation of $H$ would also separate $G-Q$. Hence $H$ is connected.

For (2), let $C$ be a component of $G[R_i]$. Since $R_i\subseteq N(q_i)$, Lemma~\ref{lem:diamond} implies that $C$ is a clique. Every vertex of $C$ is adjacent to $q_i$, so $C\cup\{q_i\}$ is also a clique. Since $Q$ is a maximum clique of order $r$, this gives $|C|\le r-1$.

Now let $C\subseteq R_i$ and $D\subseteq R_j$ be distinct components. If $i=j$, there is no edge between $C$ and $D$ by the definition of a component. Suppose $i\ne j$. We show that no vertex of $C$ can have two neighbors in $D$. Otherwise, let $x\in C$ be adjacent to distinct $y,z\in D$. The component $D$ is a clique, so $y$ and $z$ are adjacent, and because $y,z\in R_j$, both are adjacent to $q_j$. On the other hand, $x\in R_i$ and every vertex outside $Q$ has a unique neighbor in $Q$, so $x$ is not adjacent to $q_j$. Thus $x,y,z,q_j$ induce a diamond. Interchanging $C$ and $D$ gives the same conclusion for vertices of $D$. Therefore every vertex of either component is incident with at most one edge between $C$ and $D$, which is exactly to say that these edges form a matching.

For (3), let $C\subseteq R_i$ and $v\in C$. Inside $C\cup\{q_i\}$, the vertex $v$ has the $|C|-1$ other vertices of $C$ together with $q_i$ as neighbors, so it has exactly $|C|$ neighbors there. It has no neighbor in another component of $G[R_i]$. Every remaining neighbor lies in a component adjacent to $C$ in $H$, and by (2), $v$ has at most one neighbor in each such component. Consequently,
$$
d_G(v)\le |C|+d_H(C).
$$
Since $G$ is $k$-vertex-critical, $d_G(v)\ge k-1$. Rearranging gives $d_H(C)\ge k-1-|C|$, as required.
\end{proof}

Two consequences will be used in the next lemma. Since $d_G(q_i)=r-1+|R_i|$ and every vertex has degree $k-1$ or $k$, each $R_i$ has $k-r$ or $k-r+1$ vertices. Also, suppose equality holds in $d_H(C)\ge k-1-|C|$. In the proof above, we obtained this inequality from $k-1\le d_G(v)\le |C|+d_H(C)$ for every $v\in C$. Equality therefore forces $d_G(v)=k-1$ for every $v\in C$ and also forces $v$ to use every possible neighbor counted by $d_H(C)$. Thus every vertex of $C$ has one neighbor in every component adjacent to $C$ in $H$. Consequently, if adjacent components $C$ and $D$ both satisfy equality, then the matching between them saturates both sides, so $|C|=|D|$ and the matching is perfect.

\begin{lemma}\label{lem:extremal}
We have $|V(H)|\ge k-r+2$. If equality holds, then $k\in\{2r-2,2r-1\}$ and $H\cong K_r$, $K_r-e$, or $K_{r+1}-e$.
\end{lemma}

\setcounter{claim}{0}
\begin{proof}
The main obstruction comes from Gallai's theorem: two adjacent components cannot both attain equality in Lemma~\ref{lem:structure}(3) unless they are single vertices.

\begin{claim}
If adjacent components $C,D\in V(H)$ both satisfy equality in Lemma~\ref{lem:structure}(3), then $|C|=|D|=1$.
\end{claim}

\begin{proof}[Proof of Claim 1]
By the observation above, $|C|=|D|$ and the edges between them form a perfect matching. Moreover, every vertex of $C\cup D$ has degree $k-1$ in $G$. Suppose $|C|\ge2$. Then $G[C\cup D]$ consists of two cliques of the same order joined by a perfect matching. This graph is $2$-connected; deleting one vertex still leaves the two cliques connected through the remaining matching edges.

Since $G[C\cup D]$ is $2$-connected, it is contained in a single block of the subgraph induced by the degree-$(k-1)$ vertices. If $|C|=2$, then this block contains a $C_4$, so it is neither a clique nor an odd cycle. If $|C|\ge3$, then the block contains a triangle from one of the two cliques, so it is not an odd cycle, and it also contains two nonadjacent vertices from the two cliques, so it is not a clique. In either case we contradict Lemma~\ref{lem:gallai}. Therefore $|C|=|D|=1$.
\end{proof}

We now prove the lower bound. Lemma~\ref{lem:structure}(3) and $|C|\le r-1$ give $d_H(C)\ge k-r$ for every $C\in V(H)$. Since no vertex of $H$ can have degree larger than $|V(H)|-1$, this already gives $|V(H)|\ge k-r+1$. We show that equality cannot occur.

\begin{claim}
$|V(H)|\ne k-r+1$.
\end{claim}

\begin{proof}[Proof of Claim 2]
Suppose $|V(H)|=k-r+1$. Then the lower bound $d_H(C)\ge k-r$ says that every vertex of $H$ has degree $|V(H)|-1$. Thus $H$ is complete. Components contained in the same $R_i$ are nonadjacent, so a complete $H$ can contain only one component from each $R_i$. Since every $R_i$ is nonempty, there is exactly one component in each $R_i$, and hence $|V(H)|=r$. Combining this with $|V(H)|=k-r+1$ gives $k=2r-1$.

Each $R_i$ has $k-r=r-1$ or $k-r+1=r$ vertices. Since $R_i$ consists of a single component, it is a clique, and $R_i\cup\{q_i\}$ is also a clique. The maximality of $Q$ therefore rules out $|R_i|=r$, so $|R_i|=r-1$ for every $i$. Hence every component has order $r-1$. Since $H$ is complete, every component has degree $r-1=k-r=k-1-(r-1)$ and therefore satisfies equality in Lemma~\ref{lem:structure}(3). Any two components are adjacent and have order $r-1\ge3$, contradicting Claim 1.
\end{proof}

Therefore $|V(H)|\ge k-r+2$. Assume equality. Then Lemma~\ref{lem:structure}(3) gives $d_H(C)\ge k-r=|V(H)|-2$ for every component $C$, so each vertex of $H$ has at most one nonneighbor. Components contained in the same $R_i$ are nonadjacent, and therefore each $R_i$ has at most two components. Thus the rest of the proof is determined by how many of the sets $R_i$ split into two components.

\begin{claim}
At most one $R_i$ has two components.
\end{claim}

\begin{proof}[Proof of Claim 3]
Suppose $R_i$ has two components and let $C$ be one of them. The other component of $R_i$ is not adjacent to $C$, so $d_H(C)\le |V(H)|-2=k-r$. On the other hand, Lemma~\ref{lem:structure}(3) and $|C|\le r-1$ give $d_H(C)\ge k-1-|C|\ge k-r$. Thus all these inequalities are equalities. In particular, $|C|=r-1$ and $d_H(C)=k-1-|C|$. The same argument applies to the other component of $R_i$.

Now suppose two different sets $R_i$ and $R_j$ each had two components. Choose one component from each set. Each chosen component already has one nonneighbor, namely the other component in its own $R_i$ or $R_j$. Since every vertex of $H$ has at most one nonneighbor, the two chosen components must be adjacent. Both have order $r-1\ge3$ and both satisfy equality in Lemma~\ref{lem:structure}(3), contradicting Claim 1.
\end{proof}

Every $R_i$ is nonempty, and Claim 3 shows that at most one has two components. Therefore $H$ has either $r$ or $r+1$ vertices. Since we are assuming $|V(H)|=k-r+2$, we obtain $k\in\{2r-2,2r-1\}$. Thus equality leaves only two possibilities for $k$. We consider them separately and show that the fact that each component has at most one nonneighbor determines $H$ completely.

Suppose first that $k=2r-1$. Then $|V(H)|=r+1$, so exactly one $R_i$ has two components and every other $R_j$ has one. By Claim 3, the two components in $R_i$ have order $r-1$ and satisfy equality in Lemma~\ref{lem:structure}(3). Let $C=R_j$ be the unique component for some $j\ne i$. Since $|R_j|\in\{r-1,r\}$ and $R_j\cup\{q_j\}$ is a clique, maximality of $Q$ again gives $|C|=r-1$. Lemma~\ref{lem:structure}(3) now gives $d_H(C)\ge r-1$. If $d_H(C)=r-1$, then $C$ also satisfies equality in Lemma~\ref{lem:structure}(3). The two components in $R_i$ are nonadjacent to each other, while $C$ has at most one nonneighbor, so $C$ must be adjacent to at least one of them. This contradicts Claim 1. Hence $d_H(C)=r$, so $C$ is adjacent to every other component. This holds for every $j\ne i$. Therefore the only missing edge of $H$ is the edge between the two components in $R_i$, and $H\cong K_{r+1}-e$.

Now suppose $k=2r-2$. Then $|V(H)|=r$, so each of the $r$ nonempty sets $R_i$ has exactly one component. Such a component has order $k-r=r-2$ or $k-r+1=r-1$. Suppose first that some component $C$ has order $r-2$. Lemma~\ref{lem:structure}(3) gives $d_H(C)\ge k-1-|C|=r-1$, so $C$ is adjacent to every other component and satisfies equality. By Claim 1, there cannot be a second component of order $r-2$, because it would also be universal, would also satisfy equality, and would be adjacent to $C$.

If another component $D$ were not adjacent to every vertex of $H$, then $D$ must have order $r-1$. Since every vertex of $H$ has at most one nonneighbor, we would have $d_H(D)=r-2=k-1-(r-1)$, so $D$ would also satisfy equality in Lemma~\ref{lem:structure}(3). But $C$ is universal, so $C$ and $D$ are adjacent, again contradicting Claim 1. Hence every component is universal and $H\cong K_r$.

It remains to consider the case in which every component has order $r-1$. Each vertex of $H$ has at most one nonneighbor, so the missing edges form a matching. For a component $C$, Lemma~\ref{lem:structure}(3) gives $d_H(C)\ge r-2$. Since $H$ has $r$ vertices, a component satisfies equality exactly when its degree is $r-2$, that is, exactly when it is incident with a missing edge. If there were two missing edges, choose one endpoint from each. Because the missing edges form a matching, the two chosen components are adjacent. Both have order $r-1$ and both satisfy equality, contradicting Claim 1. Thus at most one edge is missing, and $H\cong K_r$ or $H\cong K_r-e$.
\end{proof}

\section{Proof of the main result}\label{sec:proof}

\setcounter{claim}{0}
\begin{proof}[Proof of Theorem~\ref{thm:main}]
Let $Q=\{q_1,\ldots,q_r\}$ be a maximum clique. Since $G$ is $k$-vertex-critical, every vertex has degree at least $k-1$. Since $G$ is noncomplete, it cannot contain $K_k$, so $r\le k-1$. Thus every vertex has degree at least $r$.

Fix $q_1\in Q$. Inside $Q$, the vertex $q_1$ has exactly $r-1$ neighbors. Since $d_G(q_1)\ge r$, it must have a neighbor $x$ outside $Q$. By Lemma~\ref{lem:diamond}, a vertex outside a maximum clique has at most one neighbor in that clique, so $q_1$ is the only neighbor of $x$ in $Q$. Therefore, for every $3\le s\le r$, if we choose any $s$-vertex subclique of $Q$ containing $q_1$, then $x$ is adjacent to exactly one vertex of that clique. The chosen clique together with $x$ consequently induces $K_s^+$.

\begin{claim}
$G$ contains an induced $L_t$ for every $3\le t\le r-1$.
\end{claim}

\begin{proof}[Proof of Claim 1]
Let $C$ be the component of $G[N(q_1)]$ containing $x$. By Lemma~\ref{lem:diamond}, $C$ is a clique. Since every vertex of $C$ is adjacent to $q_1$, the set $C\cup\{q_1\}$ is also a clique, and maximality of $Q$ gives $|C|\le r-1$.

The vertex $x$ is adjacent to every other vertex of $C$ and to $q_1$, so it has exactly $|C|$ neighbors in $C\cup\{q_1\}$. But $d_G(x)\ge r>|C|$. Hence $x$ has a neighbor $y$ outside $C\cup\{q_1\}$. We must have $yq_1\notin E(G)$: if $y$ were adjacent to $q_1$, then $x$ and $y$ would be adjacent vertices of $G[N(q_1)]$ and hence would lie in the same component $C$.

By Lemma~\ref{lem:diamond}, $y$ has at most one neighbor in $Q$. Since $y$ is not adjacent to $q_1$, any such neighbor is different from $q_1$. Fix $3\le t\le r-1$. Because we are choosing only $t$ vertices from the $r$ vertices of $Q$, we can choose a $t$-vertex subclique containing $q_1$ and avoiding the possible neighbor of $y$. On these vertices, $x$ is adjacent only to $q_1$, while $y$ is adjacent to $x$ and to none of the clique vertices. Thus the chosen clique together with $x$ and $y$ induces $L_t$.
\end{proof}

It remains to prove the final assertion. Suppose a maximum clique $Q$ is not dominating. Since a vertex-critical graph is connected, there is a vertex having no neighbor in $Q$. Choose such a vertex $y$ whose distance from $Q$ is minimum. This distance is at least two. We show that it cannot be larger than two. Suppose instead that a shortest path from $Q$ to $y$ begins $q_i-x-z-\cdots-y$. If $z$ has a neighbor in $Q$, then the portion of the path from $z$ to $y$, preceded by that neighbor, gives a shorter path from $Q$ to $y$. If $z$ has no neighbor in $Q$, then $z$ is a vertex with no neighbor in $Q$ that is closer to $Q$ than $y$, contradicting the choice of $y$. Hence the distance is exactly two, and a shortest path has the form $q_i-x-y$, where $q_i\in Q$ and $x,y\notin Q$.

The vertex $x$ is adjacent to $q_i$, and by Lemma~\ref{lem:diamond} it has no other neighbor in $Q$. By the choice of $y$, the vertex $y$ has no neighbor in $Q$. Therefore the only edges from $\{x,y\}$ to $Q$ are $xq_i$, together with the edge $xy$ of the path. Hence $Q\cup\{x,y\}$ induces $L_r$. Thus, if $G$ is $L_r$-free, every maximum clique is dominating.
\end{proof}

\begin{proof}[Proof of Theorem~\ref{thm:BK}]
Suppose the theorem is false. The two graph classes in the statement are hereditary, that is deleting vertices cannot create an induced diamond, $L_4$, or $K_5^+$. By Lemma~\ref{lem:reduction}, it is therefore enough to consider a counterexample $G_0$ in the same class with $\Delta(G_0)=9$. Because $G_0$ is a counterexample, we have 
$
\chi(G_0)>\max\{8,\omega(G_0)\}.
$
In particular, $\chi(G_0)\ge9$ and $\chi(G_0)>\omega(G_0)$. Theorem~\ref{thm:brooks} gives $\chi(G_0)\le\max\{\omega(G_0),9\}$. Since $\chi(G_0)>\omega(G_0)$, this forces $\chi(G_0)\le9$, and therefore $\chi(G_0)=9$. The same counterexample inequality then gives $\omega(G_0)\le8$.

Choose a $9$-vertex-critical induced subgraph $G$ of $G_0$. Since the class is hereditary, $G$ remains diamond-free and remains $L_4$-free or $K_5^+$-free. We have $\chi(G)=9$, $\omega(G)\le8$, and $\Delta(G)\le9$. If $\Delta(G)\le8$, then Theorem~\ref{thm:brooks} would give $\chi(G)\le\max\{\omega(G),\Delta(G)\}\le8$, contradicting $\chi(G)=9$. Hence $\Delta(G)=9$.

We may now apply Theorem~\ref{thm:BKclique}. Since $\chi(G)=\Delta(G)=9$, the graph $G$ contains $K_5$, and therefore $\omega(G)\ge5$. Let $r=\omega(G)$. Then $5\le r\le8$, so Theorem~\ref{thm:main} applies. Since $5\le r$, it gives an induced $K_5^+$; since $4\le r-1$, it also gives an induced $L_4$. Thus $G$ contains both forbidden configurations. This contradicts the assumption that $G$ is either $(\mathrm{diamond},L_4)$-free or $(\mathrm{diamond},K_5^+)$-free.
\end{proof}

\end{document}